\documentclass[12pt]{smfart}
\usepackage{color}
\usepackage{amssymb,verbatim}
\usepackage{amsthm,array,amssymb,amscd,amsfonts,amssymb,latexsym, url}
\usepackage{amsmath}
\usepackage[all]{xy}
\usepackage[french]{babel}

\newcounter{spec}
{\end{list}}

\renewcommand{\P}{{\mathbf P}}

\newcommand{\bH}{\mathbb{H}}

\newcommand{\Z}{{\mathbb Z}}
\newcommand{\Q}{{\mathbb Q}}

\newcommand{\C}{{\mathbb C}}

\renewcommand{\L}{{\mathbb L}}
\newcommand{\oi}{\hskip1mm {\buildrel \simeq \over \rightarrow} \hskip1mm}

\newcommand{\K}{{\mathcal K}}

\newcommand{\Br}{{\operatorname{Br}}}
\newcommand{\NS}{{\operatorname{NS}}}
\newcommand{\Ker}{{\operatorname{Ker}}}
\newcommand{\Coker}{{\operatorname{Coker}}}

\newcommand{\Spec}{{\operatorname{Spec }}}

\newcommand{\cd}{{\operatorname{cd}}}

\renewcommand{\lim}{\varprojlim}
\newcommand{\colim}{\varinjlim}

\numberwithin{equation}{section}

\newfont{\gothic}{eufb10}

\renewcommand{\qed}{{\hfill$\square$}}

\newtheorem{theo}{Th\'{e}or\`{e}me}[section]
\newtheorem{prop}[theo]{Proposition}

\newtheorem{lem}[theo]{Lemme}
\newtheorem{cor}[theo]{Corollaire}
\theoremstyle{definition}
\newtheorem{defi}[theo]{D\'efinition}
\theoremstyle{remark}
\newtheorem{rema}[theo]{Remarque}
\newtheorem{remas}[theo]{Remarques}

\newtheorem{exemple}[theo]{Exemple}

\newcommand{\bthe}{\begin{theo}}
\newcommand{\ble}{\begin{lem}}
\newcommand{\bpr}{\begin{prop}}
\newcommand{\bco}{\begin{cor}}
\newcommand{\bde}{\begin{defi}}
\newcommand{\ethe}{\end{theo}}
\newcommand{\ele}{\end{lem}}
\newcommand{\epr}{\end{prop}}
\newcommand{\eco}{\end{cor}}
\newcommand{\ede}{\end{defi}}

\newcommand{ \Gal}{{\rm{Gal}}}
\newcommand{\et}{{\operatorname{\acute{e}t}}}

\newcommand{\Pic}{\operatorname{Pic}}

\newcommand{\FF}{{\overline F}}

\newcommand{\X}{{\overline X}}

\DeclareFontFamily{U}{wncy}{}
\DeclareFontShape{U}{wncy}{m}{n}{%
<5>wncyr5%
<6>wncyr6%
<7>wncyr7%
<8>wncyr8%
<9>wncyr9%
<10>wncyr10%
<11>wncyr10%
<12>wncyr6%
<14>wncyr7%
<17>wncyr8%
<20>wncyr10%
<25>wncyr10}{}
\DeclareMathAlphabet{\cyr}{U}{wncy}{m}{n}

\begin{document}

\title[Descente galoisienne sur le second groupe de Chow]{Descente galoisienne sur le second groupe de Chow : mise au point et applications}
\author{Jean-Louis Colliot-Th\'el\`ene}
\address{C.N.R.S., Universit\'e Paris Sud\\Math\'ematiques, B\^atiment 425\\91405 Orsay Cedex\\France}
\email{jlct@math.u-psud.fr}

\date{24 mai 2015}

\maketitle

\begin{abstract}
Le troisi\`eme groupe de cohomologie \'etale  non ramifi\'e d'une vari\'et\'e projective et lisse,
\`a coefficients dans  les racines de l'unit\'e tordues deux fois, intervient dans plusieurs articles
r\'ecents, en particulier en relation avec le groupe de Chow de codimension 2.
Des r\'esultats g\'en\'eraux ont \'et\'e obtenus \`a ce sujet par B. Kahn en 1996.
De r\'ecents travaux,   du c\^ot\'e des groupes alg\'ebriques lin\'eaires  d'une part,
 du c\^ot\'e de la g\'eom\'etrie alg\'ebrique complexe d'autre part,
  m'incitent  \`a les passer en revue, et \`a les sp\'ecialiser 
aux vari\'et\'es  proches d'\^etre  rationnelles. 
 \end{abstract}

\begin{altabstract}
Connexions between the second Chow group of a smooth projective variety and its
third unramified cohomology group, with coefficients the roots of unity twisted twice,
feature in several recent works. In this note we revisit  a 1996 paper by B. Kahn
and specialize it to 
 nearly rational varieties. 
   \end{altabstract}

Dans tout cet article, on note 
  $F$ un corps de caract\'eristique z\'ero, $\FF$ une cl\^oture alg\'ebrique de $F$ et  $G=\Gal(\FF/F)$.
Soit  $X$ une $F$-vari\'et\'e lisse et g\'eom\'etriquement int\`egre. On note $\X=X\times_{F}{\FF}$.
On note $F(X)$ le corps des fonctions rationnelles de $X$ et $\overline{F}(X)$ le corps des fonctions
rationnelles de $\X$.

L'application naturelle entre groupes de Chow de codimension 2 
$$ CH^2(X) \to CH^2({\overline X})^G$$
n'est en g\'en\'eral ni injective ni surjective, m\^{e}me si l'on suppose 
que $X$ est projective et
que l'ensemble $X(F)$ des points rationnels de $X$ est non vide
-- \`a la diff\'erence du r\'esultat bien connu de $CH^1(X)$.

Plusieurs travaux 
ont \'et\'e consacr\'es \`a l'\'etude des noyau et conoyau de
cette application
 et aux liens entre  le groupe de Chow de codimension deux et le  troisi\`eme groupe
de cohomologie non ramifi\'ee  de $X$ \`a valeurs dans $\Q/\Z(2)$, groupe not\'e
$H^3_{nr} (X,\Q/\Z(2))$.
Citons en particulier
 \cite{CTHilb},  Raskind et l'auteur \cite{CTRaskind},
Lichtenbaum \cite{Lichtenbaum}, Kahn \cite{ Kahn0, Kahn}, C.~Voisin et l'auteur
\cite{CTVoisin}, Pirutka \cite{Pirutka}, Kahn et l'auteur \cite{CTKahn},
Merkurjev \cite{BM, merk2, merk3, merk5}, Voisin \cite{voisin}.

Une des raisons de s'int\'eresser au groupe 
 $H^3_{nr} (X,\Q/\Z(2))$ est que c'est un invariant $F$-birationnel
 des $F$-vari\'et\'es projectives et lisses, r\'eduit \`a $H^3(F,\Q/\Z(2))$
 si la $F$-vari\'et\'e $X$ est $F$-birationnelle \`a un espace projectif.

\smallskip

Le r\'esultat principal du pr\'esent article est le
 Th\'eor\`eme \ref{BMgeneralise}, qui s'applique \`a toute vari\'et\'e projective et lisse
 g\'eom\'etriquement rationnellement connexe, et qui dans le cas particulier
 des vari\'et\'es g\'eom\'etriquement rationnelles 
   \'etablit (Corollaire \ref{corprinc})
  une suite exacte  
  $$ \Ker[CH^2(X) \to CH^2({\overline X})^G] \to H^1(G,\Pic({\overline X}) \otimes {\FF}^{\times}) \to  \hskip3cm$$$$ 
   H^3_{nr} (X,\Q/\Z(2))/H^3(F,\Q/\Z(2)) 
 \to  \Coker[CH^2(X) \to CH^2({\overline X})^G]  \to $$ $$ \hskip8cm H^2(G,\Pic({\overline X}) \otimes {\FF}^{\times})$$
  sous l'une des deux hypoth\`eses suppl\'ementaires : 
 
 (i) La $F$-vari\'et\'e $X$ poss\`ede un $F$-point.
 
 (ii) La dimension cohomologique de $F$ est au plus 3.

 \medskip

D\'ecrivons la structure de l'article.

\smallskip
 
 Le \S  \ref{rappels} est consacr\'e \`a des rappels  de r\'esultats fondamentaux 
 sur la $\mathcal{K}$-cohomologie, la cohomologie non ramifi\'ee et la cohomologie
 motivique. On y rappelle aussi  (Prop.  \ref{CTKahnsuite})  un r\'esultat de \cite{CTKahn}
 apportant une correction \`a \cite{Kahn}.

\smallskip

 Au \S  \ref{uniquedivis}, sous l'hypoth\`ese que le groupe $H^0(\overline{X}, \mathcal{K}_{2})$  est uniquement divisible,
 on \'etablit par deux m\'ethodes diff\'erentes (l'une $K$-th\'eorique,
 l'autre motivique)
 une suite exacte g\'en\'erale (Propositions  \ref{longuesuitegenerale1}
et  \ref{longuesuitegenerale2}).
  On suppose ici la vari\'et\'e  $X$ lisse et g\'eom\'etriquement int\`egre, mais non n\'ecessairement propre.
Ceci s'applique en particulier aux espaces classifiants de groupes semisimples consid\'er\'es par Merkurjev
\cite{merk2}.

La premi\`ere m\'ethode, \`a l'ancienne, via la $K$-cohomologie,
est celle des articles \cite{CTRaskind}, \cite{CTVoisin}.
La seconde m\'ethode fait usage des groupes de cohomologie motivique $\Z(2)$, comme
dans l'article \cite{Kahn} de Bruno Kahn. De ce point de vue, on ne fait que
g\'en\'eraliser \cite[Thm. 1, Corollaire]{Kahn},   
avec la correction mentionn\'ee ci-dessus.
Lorsque le corps de base est  de
 dimension cohomologique au plus 1, auquel cas
la correction n'est pas utile, et lorsque de plus les vari\'et\'es consid\'er\'ees sont projectives,
ces suites exactes ont d\'ej\`a \'et\'e  utilis\'ees dans \cite{CTVoisin} et \cite{CTKahn}.

\smallskip

 Au \S  \ref{H1K2}, pour $X$ projective et lisse, on donne des conditions permettant de contr\^oler le groupe
 $H^1(\overline{X}, \mathcal{K}_{2})$ apparaissant dans les suites exactes du \S \ref{uniquedivis}.
 On donne une application aux surfaces $K3$ d\'efinies sur $\C((t))$.
 
\smallskip

 Au  \S \ref{principal}, on combine  les r\'esultats des  paragraphes pr\'ec\'edents 
 pour \'etablir les r\'esultats principaux  de l'article, le th\'eor\`eme  \ref{BMgeneralise}
 et son corollaire \ref{corprinc} cit\'e ci-dessus.

\smallskip

 Au \S  \ref{hypcubiques}, on applique les r\'esultats du  \S \ref{principal} aux hypersurfaces de Fano
 complexes.
 Pour $X \subset \P^n_{\C}$ hypersurface lisse de degr\'e $d \leq n$
    et
$F$ corps quelconque contenant $\C$, on \'etablit  $H^3_{nr}(X_{F}, \Q/\Z(2))=H^3(F,\Q/\Z(2))$  dans chacun des
 cas suivants :
 pour $n>5$ ; pour $n=5$ sous r\'eserve que l'on ait $H^3_{nr}(X, \Q/\Z(2))=0$;   pour $n=4$
 lorsqu'il existe un cycle universel de codimension 2.
On fait le lien avec
les
r\'esultats de   Auel, Parimala et l'auteur
 \cite{ACTP} et de C. Voisin \cite{voisindecdiag, voisin} sur les hypersurfaces cubiques
 et sur les cycles universels de codimension 2, r\'esultats sur lesquels on donne un nouvel \'eclairage
 -- la $K$-th\'eorie alg\'ebrique rempla\c cant certains   arguments de g\'eom\'etrie complexe (voir la
 d\'emonstration du th\'eor\`eme  \ref{h3univtrivial}).

\medskip

Par rapport \`a la premi\`ere version de cet article, mise sur arXiv en f\'evrier 2013,
cet article diff\`ere essentiellement par le contenu du pr\'esent~\S 5, motiv\'e par le travail
  \cite{ACTP}  et  par les articles \cite{voisindecdiag, voisin} de C.~Voisin.
  
  \medskip
  
  Terminons cette introduction en indiquant ce qui  n'est pas fait  dans cet article.
  
  (i) Je n'ai pas v\'erifi\'e que les arguments dans la litt\'erature utilisant les complexes
  $\Z(2)$ de Voevodsky sont compatibles avec ceux utilisant le complexe   $\Gamma(2)$ de
  Lichtenbaum ou avec ceux utilisant les groupes de cycles sup\'erieurs  de Bloch,
  dont il est fait usage dans \cite{CTKahn}. Et je n'ai pas v\'erifi\'e que dans les suites exactes
  des Propositions  \ref{longuesuitegenerale1} et  \ref{longuesuitegenerale2}, dont les
  termes sont identiques,  les fl\`eches aussi co\"{\i}ncident.
  Ceci n'affecte pas les principaux r\'esultats
  de l'article. Le lecteur v\'erifiera en effet que la Proposition \ref{longuesuitegenerale1},
  \'etablie par des m\'ethodes \`a l'ancienne via la Proposition \ref{CTKahnsuite},
 suffit \`a \'etablir tous les r\'esultats des paragraphes \ref{H1K2}, \ref{principal}, \ref{petitmotif},  \`a l'exception 
 du lemme  \ref{restricdroite} (ii), du th\'eor\`eme \ref{hypersurface} (viii) et de l'assertion de
 surjectivit\'e  de l'application $CH^2(X_{F}) \to CH^2(X_{\overline F})^G$ dans le th\'eor\`eme \ref{h3nrcubiqueC} (iii).

  (ii)  Les   longues suites  exactes des Propositions \ref{longuesuitegenerale1} 
  et  \ref{longuesuitegenerale2},
  le th\'eor\`eme  \ref{BMgeneralise}  et le corollaire \ref{corprinc} devraient 
  se sp\'ecialiser en un certain nombre des longues suites exactes  pour les vari\'et\'es classifiantes 
  de groupes alg\'ebriques lin\'eaires connexes \'etablies
  par Blinstein-Merkurjev  \cite{BM}  et par Merkurjev  \cite{merk2, merk3}. Je me suis content\'e d'allusions \`a ces
 articles en divers points du texte.

  (iii) Sur  un corps de base de caract\'eristique positive,
l'utilisation de la cohomologie de Hodge-Witt logarithmique permet de donner
des analogues de certains des r\'esultats du pr\'esent travail.
Nous renvoyons pour cela aux articles \cite{Kahn} et  \cite{CTKahn}.

\medskip

{\it Remerciements.} Cet article fait suite \`a des travaux et discussions avec 
  Bruno Kahn, et \`a des travaux de
  A. Merkurjev et de  C.~Voisin. Je remercie le rapporteur pour sa  lecture
critique du tapuscrit.

 \section{Rappels, propri\'et\'es g\'en\'erales}\label{rappels}
 
 On utilise dans cet article le complexe motivique $\Z(2)$ de faisceaux de
 cohomologie \'etale sur les vari\'et\'es lisses sur  un corps, tel qu'il a \'et\'e
d\'efini par  Lichtenbaum \cite{Lichtenbaum0, Lichtenbaum}.

    Les groupes de cohomologie \`a valeurs dans  le complexe
  $\Z(2)$  sont dans tout cet article
  les  groupes d'hypercohomologie \'etale. Ils sont not\'es $\bH^{i}(X,\Z(2))$.

  Sur un sch\'ema $X$, on note $H^{i}(X,{\mathcal K}_{j})$
 les groupes de cohomologie de Zariski \`a valeurs dans le faisceau ${\mathcal K}_{j}$
 sur $X$ associ\'e au pr\'efaisceau $U \mapsto K_{i}(H^0(U,\mathcal{O}_{X}))$, 
  o\`u la $K$-th\'eorie des anneaux est la $K$-th\'eorie de Quillen.

    \'Etant donn\'e un module galoisien $M$, c'est-\`a-dire un 
  $G$-module continu discret, on note tant\^ot $H^{i}(G,M)$ tant\^ot $H^{i}(F,M)$
  les groupes de cohomologie galoisienne \`a valeurs dans $M$.

 On note
  $\Q/\Z(2)$ le module galoisien $\colim_{n} \mu_{n}^{\otimes 2}$.

On note
$K_{3}F_{indec} : = \Coker [ K_{3}^{Milnor}F \to K_{3}^{Quillen}F].$
  
 On a les propri\'et\'es suivantes, cons\'equences de travaux de Merkurjev et Suslin \cite{MS},
de A. Suslin \cite{Suslin},  de M. Levine \cite{Levine}, de S. Lichtenbaum \cite{Lichtenbaum}, de B. Kahn  \cite{Kahn0}, \cite[Thm. 1.1, Lemme 1.4]{Kahn}.

  $\bH^{0}(F,\Z(2))=0$.
  
   $\bH^{1}(F,\Z(2))= K_{3}F_{indec}$.

    $\bH^{2}(F,\Z(2))=K_{2}F$.

     $\bH^{3}(F,\Z(2))=0$.

      $\bH^{i}(F,\Z(2))= H^{i-1}(F,\Q/\Z(2))$ si $i \geq 4$.

$\bH^{i}({\FF},\Z(2))=0$ si $i\neq 1,2$.

$\bH^{1}({\FF},\Z(2))=K_{3}({\FF})_{indec}$ est divisible, et sa torsion est  $\Q/\Z(2)$
(cf.~\cite[(1.2)]{Kahn0}).
 Il est donc extension d'un groupe uniquement divisible par
$\Q/\Z(2)$.

$\bH^2({\FF},\Z(2))=K_{2}({\FF})$ est  uniquement divisible.

\medskip

Soit $X$ une $F$-vari\'et\'e lisse g\'eom\'etriquement int\`egre, non n\'ecessairement projective. On a :

$\bH^0(X,\Z(2))=0.$

$\bH^1(X,\Z(2))=K_{3,indec}F(X).$

$\bH^1({\overline X},\Z(2))=K_{3,indec}{\FF}(X)$ 
  est extension d'un groupe uniquement divisible
par $\Q/\Z(2)$. Ceci r\'esulte de  la suite exacte \cite[(1.2)]{Kahn0}
et de \cite[Thm. 3.7]{Suslin}).

$\bH^2(X,\Z(2))=H^0(X,{\mathcal K}_{2}).$

$\bH^3(X,\Z(2))= H^1(X,{\mathcal K}_{2}).$ 

On  a  la suite exacte fondamentale (Lichtenbaum, Kahn \cite[Thm. 1.1]{Kahn})
\begin{equation}\label{fondam}
  0 \to CH^2(X) \to \bH^4(X,\Z(2)) \to H^3_{nr}(X,\Q/\Z(2)) \to 0 
\end{equation}
o\`u 
$$H^3_{nr}(X,\Q/\Z(2)) = H^0(X,{\mathcal H}^3(X,\Q/\Z(2)))$$
est le sous-groupe de $H^3(F(X),\Q/\Z(2)) $
form\'e des \'el\'ements non ramifi\'es en tout point de codimension 1 de $X$.

\medskip

 Pour toute $F$-vari\'et\'e   projective,  lisse et  g\'eom\'etriquement int\`egre $X$, dans l'article  \cite{CTRaskind}
 avec W. Raskind,  on a \'etabli  que les groupes $H^{0}({\overline X},{\mathcal K}_{2})$ et $H^{1}({\overline X},{\mathcal K}_{2})$
 sont chacun  extension d'un groupe fini  par un groupe divisible.
Si la  dimension cohomologique de $F$ satisfait $\cd(F) \leq i$, ceci implique que 
  les groupes de cohomologie galoisienne
 $H^{r}(G,H^{0}({\overline X},{\mathcal K}_{2}))$ et $H^{r}(G,H^{1}({\overline X},{\mathcal K}_{2}))$
 sont nuls   pour $r \geq i+1$.

 \medskip

On a une suite spectrale
 $$E_{2}^{pq} = H^p(G,\bH^q({\overline X},\Z(2))) \Longrightarrow \bH^n(X,\Z(2)).$$

 \begin{rema}\label{remarqueK2G}
 Pour $X=\Spec(F)$, compte tenu des identifications ci-dessus, cette suite spectrale
 donne une suite exacte
 $$  H^1(G,\Q/\Z(2)) \to K_{2}F \to K_{2}{\FF}^G \to H^2(G,\Q/\Z(2)) \to 0.$$
 Ceci est un cas particulier de \cite[Thm. 2.1]{Kahn0}.
 \end{rema}

En comparant la suite exacte fondamentale (\ref{fondam}) au niveau $F$ et au niveau~${\FF}$,
 en prenant les points fixes de $G$ agissant sur la suite
au niveau ${\FF}$, et en utilisant le lemme du serpent, on obtient :

\begin{prop}\label{longuesuitefondamentale}
Soit $X$ une $F$-vari\'et\'e lisse et g\'eom\'etriquement int\`egre.
Soit $\varphi: \bH^4(X,\Z(2)) \to \bH^4({\overline X},\Z(2))^G$.
On a alors  
une suite exacte
 $$0 \to \Ker[CH^2(X) \to CH^2({\overline X})^G] \to \Ker(\varphi) \to \hskip3cm $$
 $$ \Ker[H^3_{nr}(X,\Q/\Z(2))\to H^3_{nr}({\overline X},\Q/\Z(2))] \to
  $$ $$
 \hskip4cm \Coker[CH^2(X) \to CH^2({\overline X})^G] \to  \Coker(\varphi) .
 $$
 \end{prop}

Notons
\begin{equation}\label{N(X)}
  {\mathcal N}(X) : = 
\Ker \Big[ H^2(G, K_{2}({\FF}(X))  \to H^2(G,\bigoplus_{x \in {\overline X}^{(1)} }{\FF}(x)^{\times})  \Big]
\end{equation}

L'\'enonc\'e suivant est essentiellement \'etabli dans \cite{CTKahn}.

\begin{prop}\label{CTKahnsuite}
Soit $X$ une $F$-vari\'et\'e lisse et  g\'eom\'etriquement int\`egre.

(a) On a une suite exacte
$$H^3(F,\Q/\Z(2)) \to 
\Ker[ H^3_{nr}(X,\Q/\Z(2))  \to  H^3_{nr}({\overline X},\Q/\Z(2))] \to \hskip2cm$$ $$ \hskip3cm
  {\mathcal N}(X) \to \Ker[H^4(F,\Q/\Z(2)) \to H^4(F(X),\Q/\Z(2))].$$

(b) Si $X(F) \neq \emptyset $ ou si $\cd(F) \leq 3$, 
on a un isomorphisme  
$$\Ker[ H^3_{nr}(X,\Q/\Z(2))/H^3(F,\Q/\Z(2)) \to  H^3_{nr}({\overline X},\Q/\Z(2))]
\oi  {\mathcal N}(X).$$

(c) Si $X$ est de dimension au plus 2, on a
une suite exacte
$$H^3(F,\Q/\Z(2)) \to 
 H^3_{nr}(X,\Q/\Z(2) \to  {\mathcal N}(X) \to H^4(F,\Q/\Z(2)) \to H^4(F(X),\Q/\Z(2)).$$
\end{prop}

\begin{proof} L'\'enonc\'e (a) est  \cite[Prop. 6.1, Prop. 6.2]{CTKahn}.  L'\'enonc\'e (b) est
une cons\'equence facile  de (a).
La proposition 6.1 de \cite{CTKahn}
  montre aussi que, si $X$ est de dimension au plus 2, alors
le complexe
 $$ H^3(F,\Q/\Z(2) \to   \Ker[H^3_{nr}(X,\Q/\Z(2))  \to  H^3_{nr}({\overline X},\Q/\Z(2))]  \to  \hskip3cm$$$$ 
  \hskip3cm
 {\mathcal N}(X)
 \to H^4(F,\Q/\Z(2)) \to H^4(F(X),\Q/\Z(2))$$
 est une suite exacte
  $$ H^3(F,\Q/\Z(2) \to   H^3_{nr}(X,\Q/\Z(2))   \to  \hskip6cm $$$$ \hskip3cm {\mathcal N}(X)
 \to H^4(F,\Q/\Z(2)) \to H^4(F(X),\Q/\Z(2)).$$
  En effet les groupes $H^3(A_{s},\Q/\Z(2))$ intervenant dans la proposition~6.1 de \cite{CTKahn}
  sont alors nuls  : via la conjecture de Gersten, cela r\'esulte du fait que
le corps des fractions de $A_{s}$ est de dimension cohomologique 2,
si bien que le complexe de la proposition 6.1 de \cite{CTKahn}  est alors  exact.
\end{proof}

\section{Le cas o\`u le groupe $H^0({\overline X},{\mathcal K}_{2})$ est uniquement divisible}\label{uniquedivis}

Le but de ce pagragraphe est d'\'etablir la  proposition
\ref{longuesuitegenerale1}. On le fait  d'abord  
par une m\'ethode  ``$K$-th\'eorique''  (paragraphe \ref{methK})
qui  se pr\^ete plus aux calculs explicites
des fl\`eches intervenant dans les suites exactes. 
 La version ``motivique'' (paragraphe \ref{methM})
est plus souple quand il s'agit d'\'etudier la fonctorialit\'e en la $F$-vari\'et\'e  $X$
des suites concern\'ees.

\medskip

Dans ce paragraphe, on consid\`ere une $F$-vari\'et\'e
$X$ lisse et g\'eom\'e\-tri\-quement int\`egre, telle que le groupe  $H^0({\overline X},{\mathcal K}_{2})$ est uniquement divisible,
mais on ne suppose pas $X$ projective.

\subsection{M\'ethode  $K$-th\'eorique}\label{methK}

Pour $i\geq 1$, les fl\`eches naturelles 
$$ H^{i}(G,K_{2}{\FF}(X)) \to  H^{i}(G,K_{2}{\FF}(X)/K_{2}{\FF}) \to H^{i}(G,K_{2}{\FF}(X)/H^0({\overline X},{\mathcal K}_{2})   ) $$
sont  alors des isomorphismes.

D'apr\`es un th\'eor\`eme de Quillen (conjecture de Gersten pour la  $K$-th\'eorie),
le complexe
$$ K_{2}{\FF}(X) \to \bigoplus_{x\in {\overline X}^{(1)}} {\FF}(x)^{\times} \to \bigoplus_{x \in {\overline X}^{(2)}} \Z$$
est  le complexe des sections globales d'une r\'esolution flasque du faisceau ${\mathcal K}_{2}$
sur la ${\FF}$-vari\'et\'e lisse ${\overline X}$.

Ce complexe donne donc des suites exactes courtes de modules galoisiens
$$0 \to K_{2}{\FF}(X)/H^0({\overline X},{\mathcal K}_{2})Ê\to Z \to H^1({\overline X},{\mathcal K}_{2}) \to 0$$
$$0 \to Z \to \oplus_{x\in \X^{(1)}} {\FF}(x)^{\times}  \to I \to 0$$
$$0 \to I \to  \oplus_{x \in \X^{(2)}} \Z \to CH^2({\overline X}) \to 0.$$

En utilisant le th\'eor\`eme 90 de Hilbert et le lemme de Shapiro, 
le th\'eor\`eme de Merkurjev--Suslin et en particulier sa cons\'equence
  \cite[Thm.~1]{CTHilb} \cite[1.8]{Suslin}
$$K_{2}F(X)/K_{2}F = (K_{2}{\FF}(X)/K_{2}{\FF})^G,  $$
par des arguments classiques (cf. \cite{CTRaskind, CTVoisin}) de cohomologie galoisienne, 
 on obtient :

\begin{prop}\label{alancienne0}
Soit $X$ une $F$-vari\'et\'e lisse et  g\'eom\'etriquement int\`egre telle que le
groupe $H^0({\overline X},{\mathcal K}_{2})$ soit uniquement divisible.
Soit $ {\mathcal N}(X)$ comme en  (\ref{N(X)}).
On a alors une suite exacte 
$$ 0 \to H^1(X,{\mathcal K}_{2}) \to H^1({\overline X},{\mathcal K}_{2})^G \to \hskip5cm $$
$$ H^1(G, K_{2}{\FF}(X)) \to \Ker [CH^2(X) \to CH^2({\overline X})] \to $$$$
\hskip2cm H^1(G, H^1(\X,{\mathcal K}_{2})) \to {\mathcal N}(X) \to $$
$$\hskip3cm  \Coker[CH^2(X) \to CH^2({\overline X})^G]
\to H^2(G, H^1({\overline X},{\mathcal K}_{2})).$$
\end{prop}

 \medskip
  
 Pour  toute
 $F$-vari\'et\'e $X$  g\'eom\'etriquement int\`egre,
 un th\'eor\`eme de B.~Kahn  \cite[Cor. 2, p. 70]{Kahn0} donne un isomorphisme
$$ H^1(G, K_{2}{\FF}(X) ) \oi \Ker [H^3(F,\Q/\Z(2)) \to H^3(F(X),\Q/\Z(2))].$$

On a donc \'etabli :

\begin{prop}\label{alancienne}
Soit $X$ une $F$-vari\'et\'e lisse et g\'eom\'etriquement int\`egre telle que le
groupe $H^0({\overline X},{\mathcal K}_{2})$ soit uniquement divisible.
Soit $ {\mathcal N}(X)$ comme en  (\ref{N(X)}).
On a alors une suite exacte  
$$ 0 \to H^1(X,{\mathcal K}_{2}) \to H^1({\overline X},{\mathcal K}_{2})^G \to 
\hskip6cm $$
$$
\Ker [H^3(F,\Q/\Z(2)) \to H^3(F(X),\Q/\Z(2))] \to \hskip3cm $$$$
\hskip1cm
\Ker [CH^2(X) \to CH^2({\overline X})] \to  
H^1(G, H^1(\X,{\mathcal K}_{2})) \to {\mathcal N}(X) \to 
$$$$\hskip4cm
  \Coker[CH^2(X) \to CH^2({\overline X})^G]
\to H^2(G, H^1({\overline X},{\mathcal K}_{2})).$$
\end{prop}

\begin{rema}
Soit $X$ un espace principal homog\`ene d'un $F$-groupe semisimple simplement connexe
absolument presque simple. On a $K_{2}(\overline{F}) =H^0({\overline X},{\mathcal K}_{2})$,
et ce groupe est donc uniquement divisible. 
On a par ailleurs
  $H^1(\overline{X},{\mathcal K}_{2})=\Z$
avec action triviale du groupe de Galois.
L'image de $1$  
par l'application $$H^1({\overline X},{\mathcal K}_{2})^G  \to H^3(F,\Q/\Z(2))$$
est (au signe pr\`es) l'invariant de Rost de $X$. Pour tout ceci, voir \cite[Part II, \S 6]{GMS}.
\end{rema}

En combinant  les propositions \ref{alancienne} et \ref{CTKahnsuite}
on trouve :

\begin{prop}\label{longuesuitegenerale1}
Soit $X$ une $F$-vari\'et\'e lisse et g\'eom\'etriquement int\`egre.
Supposons le groupe $H^0({\overline X},{\mathcal K}_{2})$ uniquement divisible.
Sous l'une des hypoth\`eses
$X(F) \neq \emptyset$ ou $\cd(F) \leq 3$, on a
une  suite exacte
 $$0  \to  H^1(X,{\mathcal K}_{2}) \to H^1({\overline X},{\mathcal K}_{2})^G \to  \hskip6cm$$ 
 $$  \Ker[H^3(F,\Q/\Z(2)) \to H^3(F(X),\Q/\Z(2))] \to  \hskip2cm $$$$
\hskip2cm   \Ker[CH^2(X) \to CH^2({\overline X})^G] \to 
    H^1(G,H^1({\overline X},{\mathcal K}_{2})) \to  \hskip2cm$$$$
    \hskip1cm 
    \Ker[H^3_{nr}(X,\Q/\Z(2))/H^3(F,\Q/\Z(2))
 \to H^3_{nr}({\overline X},\Q/\Z(2))]  \to
$$ $$  \hskip3cm
  \Coker[CH^2(X) \to CH^2({\overline X})^G] \to H^2(G,H^1({\overline X},{\mathcal K}_{2})).$$
  Sous l'hypoth\`ese $X(F) \neq \emptyset$, le groupe $  \Ker[H^3(F,\Q/\Z(2)) \to H^3(F(X),\Q/\Z(2))]$
  est nul.
 \end{prop}

\begin{rema}
Lorsque l'on suppose $K_{2}(\overline{F}) =H^0({\overline X},{\mathcal K}_{2})$ 
et $H^3_{nr}({\overline X},\Q/\Z(2))=0$, on retrouve l'\'enonc\'e de B. Kahn
\cite[Thm. 1, Corollaire]{Kahn}.
\end{rema}

 \subsection{M\'ethode  motivique}\label{motivique}\label{methM}

Toujours {\it sous l'hypoth\`ese
que le groupe  $H^0({\overline X},K_{2}) \simeq \bH^2({\overline X},\Z(2))$ est uniquement divisible}, \'etudions  la suite spectrale 
$$E_{2}^{pq} = H^p(G,\bH^q({\overline X},\Z(2))) \Longrightarrow \bH^n(X,\Z(2)).$$

La page $E_{2}^{pq}$ contient un certain nombre de z\'eros. Tous les termes $E_{2}^{p0}$ sont nuls.
 Comme   $ \bH^2({\overline X},\Z(2))$ est suppos\'e uniquement divisible,
 tous les termes $E_{2}^{p2}=H^p(G, \bH^2({\overline X},\Z(2)))$ pour $p \geq 1$ sont nuls.
 Les termes $E_{2}^{p1}$ sont \'egaux \`a $H^p(F,\Q/\Z(2))$ pour $p\geq 2$,
 groupe qui co\"{\i}ncide avec $H^{p+1}(F,\Z(2))$ pour $p\geq 3$. 
 La fl\`eche $E_{2}^{02} \to E_{2}^{21}$, soit
 $H^0({\overline X},{\mathcal K}_{2})^G \to H^2(F,\Q/\Z(2))$, est surjective, car il en est d\'ej\`a ainsi
 de $K_{2}{\FF}^G \to H^2(F,\Q/\Z(2))$ (Remarque \ref{remarqueK2G}).
  
 Notons comme ci-dessus $\varphi: \bH^4(X,\Z(2)) \to \bH^4({\overline X},\Z(2))^G$.
 L'analyse de la suite spectrale donne les \'enonc\'es suivants.
 
 \medskip
 
1) Il y a une suite exacte
 $$ 0  \to \bH^3(X,\Z(2)) \to  (\bH^3({\overline X},\Z(2))^G \to  \bH^4(F,\Z(2)) \to   \Ker(\varphi) \to $$
 $$  \hskip2cm H^1(G,H^1({\overline X},{\mathcal K}_{2})) \to 
 \Ker [\bH^5(F,\Z(2)) \to \bH^5(X,\Z(2))].$$
 
 Ainsi il y a une suite exacte
  $$ 0  \to H^1(X,\K_{2})  \to (H^1({\overline X},{\mathcal K}_{2}))^G \to  H^3(F,\Q/\Z(2)) \to \Ker(\varphi) \to  \hskip1cm $$
 $$   \hskip2cm  H^1(G,\bH^3({\overline X},\Z(2))) \to 
 \Ker[H^4(F,\Q/\Z(2)) \to H^4(F(X) ,\Q/\Z(2))].$$

 En particulier, si  $X(F) \neq \emptyset$ ou si l'on a  $\cd(F) \leq 3$,   alors on a une suite exacte
  $$ 0  \to H^1(X,{\mathcal K}_{2})  \to (H^1({\overline X},{\mathcal K}_{2}))^G \to  \hskip6cm$$
 $$  \hskip3cm H^3(F,\Q/\Z(2))  \to \Ker(\varphi)  \to H^1(G,H^1({\overline X},\K_{2})) \to 0.$$
 La fl\`eche $H^3(F,\Q/\Z(2)) \to \Ker(\varphi)$ est injective 
  si $X(F)\neq \emptyset$, ou si $H^3(F,\Q/\Z(2))$  est nul, par exemple si $\cd(F)\leq 2$.

 \medskip

2) Pour le conoyau de $\varphi$, on trouve une suite exacte
 $$ 0 \to D \to  \Coker(\varphi) \to  H^2(G,H^1({\overline X},{\mathcal K}_{2}))$$
 o\`u $D$ est un sous-quotient de 
 $\Ker [\bH^5(F,\Z(2)) \to \bH^5(X,\Z(2))].$
 Ce dernier groupe est nul si le noyau de
 $H^4(F,\Q/\Z(2))  \to H^4(F(X),\Q/\Z(2))$ est nul.
 En particulier $D=0$ si $X(F)\neq \emptyset$,
ou si
$\bH^5(F,\Z(2)) = H^4(F,\Q/\Z(2))$ est nul, par exemple si $\cd(F) \leq 3$.

 \medskip
 
 En utilisant la proposition \ref{longuesuitefondamentale}, on voit
que  pour  toute  $F$-vari\'et\'e $X$ lisse g\'eom\'etriquement int\`egre 
 avec $H^0({\overline X},{\mathcal K}_{2})$ uniquement divisible,
sous l'hypoth\`ese que  soit $X(F)\neq \emptyset$ soit  $\cd(F)\leq 3$, on a
une suite exacte
   $$0 \to \Ker[CH^2(X) \to CH^2({\overline X})^G] \to \Ker(\varphi) \to  \hskip3cm   $$$$\Ker[H^3_{nr}(X,\Q/\Z(2))\to H^3_{nr}({\overline X},\Q/\Z(2))] \to
$$ $$
\hskip3cm   \Coker[CH^2(X) \to CH^2({\overline X})^G] \to H^2(G,H^1({\overline X},{\mathcal K}_{2})).$$
 et une suite exacte
 $$ H^3(F,\Q/\Z(2)) \to \Ker(\varphi)  \to H^1(G,H^1({\overline X},{\mathcal K}_{2})) \to 0.$$

Si  l'on quotiente  les deux termes $\Ker(\varphi) \subset \bH^4(X,\Z(2))$ et $ H^3_{nr}(X,\Q/\Z(2))$
par l'image de  $ \bH^4(F,\Z(2)) \simeq  H^3(F,\Q/\Z(2))$,  
ce qui
par fonctorialit\'e de la suite spectrale
appliqu\'ee au morphisme structural
$X \to \Spec(F)$ induit une fl\`eche
$\Ker(\varphi)/ \bH^4(F,\Z(2)) \to H^3_{nr}(X,\Q/\Z(2))/H^3(F,\Q/\Z(2))$,
on trouve :

\begin{prop}\label{longuesuitegenerale2}
Soit $X$ une $F$-vari\'et\'e lisse et g\'eom\'etriquement int\`egre. Supposons
que   $H^0({\overline X},K_{2})$ est uniquement divisible.
Supposons en outre  que l'on a $X(F) \neq \emptyset$ ou $\cd(F) \leq 3$.
On a alors une suite exacte
   $$0  \to  H^1(X,{\mathcal K}_{2}) \to H^1({\overline X},{\mathcal K}_{2})^G \to  \hskip6cm $$$$ 
   \Ker[H^3(F,\Q/\Z(2)) \to H^3(F(X),\Q/\Z(2))]  \to  \hskip3cm$$$$
     \Ker[CH^2(X) \to CH^2({\overline X})^G] \to 
    H^1(G,H^1({\overline X},{\mathcal K}_{2})) \to $$$$ 
     \hskip3cm
 \Ker[H^3_{nr}(X,\Q/\Z(2))/H^3(F,\Q/\Z(2))
 \to H^3_{nr}({\overline X},\Q/\Z(2))] \to
$$ $$ \hskip6cm
  \Coker[CH^2(X) \to CH^2({\overline X})^G] \to H^2(G,H^1({\overline X},{\mathcal K}_{2})).$$
  Sous l'hypoth\`ese $X(F) \neq \emptyset$, le groupe $  \Ker[H^3(F,\Q/\Z(2)) \to H^3(F(X),\Q/\Z(2))]$
  est nul.
 \end{prop}
  
  \begin{remas}
  
  (a) La d\'emonstration n'utilise ni le groupe  ${\mathcal N}(X)$  d\'efini en (\ref{N(X)})
   ni
  la  proposition \ref{CTKahnsuite}.
  
 (b)  L'\'enonc\'e de cette proposition est identique  \`a celui de la proposition
   \ref{longuesuitegenerale1}, mais il n'est pas clair a priori que les fl\`eches
   intervenant dans ces deux suites exactes co\"{\i}ncident.
  \end{remas}

\subsection{Comparaison entre les deux m\'ethodes}\label{comparaison}

Supposons   $H^0({\overline X},{\mathcal K}_{2})$ uniquement divisible.
On a une
  suite  exacte
$$\Ker [CH^2(X) \to CH^2({\overline X})]  \to
H^1(G, H^1(\X,{\mathcal K}_{2})) \buildrel{\rho} \over{\rightarrow} \hskip3cm $$
$$  \hskip6cm {\mathcal N}(X) \to  
  \Coker[CH^2(X) \to CH^2({\overline X})^G]$$
extraite de la proposition \ref{alancienne}, et utilis\'ee dans la d\'emonstration
de la proposition   \ref{longuesuitegenerale1}.
On a une
 suite exacte
$$\Ker [CH^2(X) \to CH^2({\overline X})] \to  
\Ker(\varphi)  \buildrel{\sigma} \over{\rightarrow}  \hskip6cm $$$$ \hskip1cm \Ker[H^3_{nr}(X,\Q/\Z(2)) 
 \to H^3_{nr}({\overline X},\Q/\Z(2))] \to
  \Coker[CH^2(X) \to CH^2({\overline X})^G]$$
 extraite de la proposition \ref{longuesuitefondamentale} 
 et utilis\'ee dans la d\'emonstration de la proposition
  \ref{longuesuitegenerale2}.
  Les termes de gauche et de droite dans ces deux suites exactes co\"{\i}ncident.
Sous r\'eserve de v\'erification des commutativit\'es des  diagrammes, 
sur tout corps
$F$ (sans condition de dimension cohomologique), le lien entre ces deux suites
est fourni par le diagramme  de suites exactes verticales

\[\xymatrix{
H^3(F,\Q/\Z(2)) \ar[d] \ar[r]^{=}  & H^3(F,\Q/\Z(2))  \ar[d]     \\
\Ker(\varphi)  \ar[r]^{\sigma}  \ar[d] & \Ker[H^3_{nr}(X,\Q/\Z(2))
\to H^3_{nr}({\overline X},\Q/\Z(2))] \ar[d] \\
H^1(G,H^1({\overline X},{\mathcal K}_{2})) \ar[d]  \ar[r]^{\rho}  & 
{\mathcal N}(X) \ar[d] \\
\Ker[H^4(F,\Q/\Z(2) \to
  H^4(F(X),\Q/\Z(2))]
   \ar[r]^{=}  &  \Ker[H^4(F,\Q/\Z(2) \to
  H^4(F(X),\Q/\Z(2))]
}\]
o\`u  la suite verticale de droite vaut pour toute $F$-vari\'et\'e lisse et  g\'eom\'e\-tri\-quement int\`egre $X$
(\cite{CTKahn}, voir   la proposition \ref{CTKahnsuite} ci-dessus),
et o\`u celle de gauche est \'etablie au d\'ebut de la section \ref{motivique}  pour les $F$-vari\'et\'es~$X$ telles que $H^0({\overline X},\K_{2})$
est uniquement divisible.

\subsection{Vari\'et\'es avec $H^0({\overline X},{\mathcal K}_{2})$ uniquement divisible}

 \subsubsection{Les espaces classifiants de groupes semisimples}\label{classifiants}
 Soit $H$ un $F$-groupe semisimple connexe,   soit $V$ une repr\'esentation lin\'eaire g\'en\'eriquement libre de $H$
 poss\'edant un ouvert $H$-stable
  $U \subset V$,  de compl\'ementaire un ferm\'e de codimension au moins 3 dans $V$,
  et que de plus l'on dispose d'une application quotient $U \to U/H $ qui soit un $H$-torseur.
  Soit $X:=U/H$. Soit $H_{sc}$ le rev\^etement simplement connexe de $H$
  et soit $C$ le noyau de l'isog\'enie $H_{sc} \to H$, puis $\hat{C}$ le module galoisien fini
  d\'efini par son groupe des caract\`eres. Comme le montre Merkurjev dans \cite[Thm. 5.3]{merk2},
  on a des identifications
  $$ K_{2}(\overline{F}) = H^0({\overline X},{\mathcal K}_{2})$$
  et
  $$\hat{C}(1):= {\rm Tor}_{1}^{\Z} (\hat{C}, \Q/\Z(1)) \simeq  H^1({\overline X},{\mathcal K}_{2}).$$
  Le groupe $ K_{2}(\overline{F})$ est uniquement divisible. La $F$-vari\'et\'e $X$ poss\`ede un point
  $F$-rationnel.
  
 La proposition \ref{longuesuitegenerale1} et la proposition  \ref{longuesuitegenerale2}
donnent donc chacune une suite exacte longue
  $$0  \to   
     \Ker[CH^2(X) \to CH^2({\overline X})^G] \to 
    H^1(G,\hat{C}(1)) \to \hskip2cm $$$$ 
     \hskip1cm
 \Ker[H^3_{nr}(X,\Q/\Z(2))/H^3(F,\Q/\Z(2))
 \to H^3_{nr}({\overline X},\Q/\Z(2))] \to
$$ $$ \hskip5cm
  \Coker[CH^2(X) \to CH^2({\overline X})^G] \to H^2(G,\hat{C}(1)).$$
  
Il serait int\'eressant de d\'eterminer le lien  entre la suite exacte \`a 5 termes 
obtenue par Merkurjev \cite[Thm. 3.9]{merk3} 
et les suites exactes \`a 5 termes  ci-dessus. Elles ont en commun
leurs deux premiers termes, et leur dernier terme.

\subsubsection{Vari\'et\'es projectives}

 Pour  ${\FF}$ un corps alg\'ebriquement clos -- toujours suppos\'e de caract\'eristique nulle --
 et $Y$ une ${\FF}$-vari\'et\'e int\`egre, projective et   lisse, les propri\'et\'es suivantes sont \'equivalentes :
 
 (i)  Le groupe de Picard $\Pic(Y)$ est sans torsion.
 
 (ii) Pour tout entier $n>0$, $H^1_{\et}(Y,\mu_{n})=0$.
 
 (iii)  $H^1(Y,\mathcal{O}_{Y})=0$ et le groupe de N\'eron--Severi $\NS(Y)$ est sans torsion.
 
 (iv) Le groupe $H^0(Y,{\mathcal K}_{2})$ est uniquement divisible.
 
 L'\'equivalence des trois premi\`eres propri\'et\'es est classique.
 Pour l'\'equivalence avec la quatri\`eme, voir
 \cite[Prop. 1.13]{CTRaskind},  qui s'appuie sur des r\'esultats
 de Merkurjev et de Suslin.
 
 \medskip
 
 Les propri\'et\'es ci-dessus sont satisfaites  par toute ${\FF}$-vari\'et\'e projective et lisse
g\'eom\'etriquement unirationnelle, mais aussi par toute surface $K3$ et par toute surface
projective et lisse dans l'espace projectif ${\bf P}^3$.

Pour une ${\FF}$-surface $Y$ projective et lisse satisfaisant ces propri\'et\'es, 
 la dualit\'e de Poincar\'e implique la nullit\'e
des groupes $H^3_{\et}(Y,\mu_{n})$ pour tout $n>0$. On sait (Bloch, Merkurjev--Suslin, cf. \cite[(2.1)]{CTRaskind})
que la  nullit\'e de ces  groupes
implique que le groupe de Chow  $CH^2(Y)$  n'a pas de torsion.

Pour une $F$-surface $X$ projective, lisse et  g\'eom\'etriquement  int\`egre telle que ${\overline X}$
satisfasse ces propri\'et\'es, le groupe
$\Ker[CH^2(X) \to CH^2({\overline X})]$ co\"{\i}ncide donc avec le sous-groupe
de torsion $CH^2(X)_{tors}$ de $CH^2(X)$.

\medskip

Sans hypoth\`ese suppl\'ementaire sur $X$, 
il est   difficile de contr\^oler le module galoisien $H^1({\overline X},K_{2})$ et
 l'application
$$CH^2(X)_{tors} = \Ker[CH^2(X) \to CH^2({\overline X})] \to H^1(G,H^1({\overline X},{\mathcal K}_{2})).$$
 Renvoyons ici le lecteur au d\'elicat  travail d'Asakura et Saito \cite{AS} qui \'etablit que pour
un corps $p$-adique $F$ et une 
surface lisse dans  $\P^3_{F}$, de degr\'e au moins 5   ``tr\`es g\'en\'erale'',
le   groupe $$CH^2(X)_{tors} \subset H^1(G,H^1({\overline X},{\mathcal K}_{2}))$$ est infini.

Au paragraphe suivant, on donnera des hypoth\`eses  restrictives
 permettant de facilement contr\^oler  le module $H^1({\overline X},{\mathcal K}_{2})$ et sa cohomologie galoisienne.

\section{Le module galoisien  $H^1({\overline X},{\mathcal K}_{2})$}\label{H1K2}

On consid\`ere la fl\`eche naturelle
$$ \Pic({\overline X})\otimes {\FF}^{\times} \to H^1({\overline X},{\mathcal K}_{2}).$$

\begin{prop}\label{H2Giso}
Soit  $X$ une $F$-vari\'et\'e projective, lisse et g\'eom\'etri\-que\-ment int\`egre.
Supposons $H^{2}(X,\mathcal{O}_{X})=0$   et supposons que
 les groupes $H^{3}_{\et}({\overline X},\Z_{\ell})$ sont sans torsion.
 Alors pour tout $i \geq 2$,  la fl\`eche 
$$H^{i}(G,\Pic({\overline X})\otimes {\FF}^{\times}) \to H^{i}(G,H^1({\overline X},{\mathcal K}_{2}))$$
est un isomorphisme.
\end{prop}

\begin{proof} D'apr\`es \cite[Thm. 2.12]{CTRaskind},  la fl\`eche  Galois \'equivariante
 $$ \Pic({\overline X})\otimes {\FF}^{\times} \to H^1({\overline X},{\mathcal K}_{2})$$
 a alors noyau et conoyau uniquement divisibles.
 Elle induit donc un isomorphisme sur $H^{i}(G,\bullet)$ pour $i \geq 2$.  \end{proof}

  \begin{rema}
L'hypoth\`ese que les groupes $H^{3}_{\et}({\overline X},\Z_{\ell})$  sont sans torsion
 est \'equivalente \`a l'hypoth\`ese que le groupe de Brauer $\Br({\overline X})$ est
 un groupe divisible. 
 \end{rema} 
 
\begin{prop}\label{H1Giso}
Soit $X$ une $F$-vari\'et\'e projective, lisse et  g\'eom\'e\-tri\-quement int\`egre.
Supposons qu'il existe  une courbe $V \subset X$ telle que sur un domaine universel $\Omega$
l'application $CH_{0}(V_{\Omega}) \to  CH_{0}(X_{\Omega})$ est surjective, et supposons
que les groupes $H^{3}_{\et}({\overline X},\Z_{\ell})$  sont sans torsion.
 Alors pour tout $i \geq 1$,  la fl\`eche 
$$H^{i}(G,\Pic({\overline X})\otimes {\FF}^{\times}) \to H^{i}(G,H^1({\overline X},{\mathcal K}_{2}))$$
est un isomorphisme.
\end{prop}
\begin{proof}
 D'apr\`es \cite[Thm. 2.12; Prop. 2.15]{CTRaskind},  sous les hypoth\`eses
de la proposition, 
 la fl\`eche  Galois-\'equivariante
 $$ \Pic({\overline X})\otimes {\FF}^{\times} \to H^1({\overline X},{\mathcal K}_{2})$$
 est surjective et 
 a un noyau uniquement divisible.
 Elle induit donc un isomorphisme sur $H^{i}(G,\bullet)$ pour $i \geq 1$. 
 \end{proof}

\begin{remas}
 
Rappelons que l'on suppose ${\rm car}(F)=0$.

 (a)   L'hypoth\`ese sur le groupe de Chow des z\'ero-cycles faite dans la proposition \ref{H1Giso}
 implique $H^{i}(X,\mathcal{O}_{X})=0$ pour $i \geq 2$. Elle implique que le groupe
 de Brauer $\Br({\overline X})$ est un groupe fini.
  Elle est satisfaite pour
 les vari\'et\'es ${\overline X}$ domin\'ees rationnellement par le produit d'une courbe et d'un espace projectif,
 en particulier elle est satisfaite pour les vari\'et\'es g\'eom\'etriquement unirationnelles.
 
 (b)   
 Sous les hypoth\`eses de la proposition \ref{H1Giso}, 
 on a    $\Br({\overline X})=0$.
  
 (c) Toutes les hypoth\`eses de la proposition \ref{H1Giso}  sont satisfaites pour une vari\'et\'e ${\overline X}$
 qui est    facteur direct birationnel d'une vari\'et\'e rationnelle.

\end{remas}

La proposition suivante (cf. \cite[Prop. 8.10]{CTVoisin}) s'applique  par exemple aux surfaces $K3$
sur $F$ corps de  fonctions d'une variable sur $\C$, ou sur $F=\C((t))$. 

\begin{prop}  Supposons le corps  $F$ de dimension cohomologique au plus 1. Soit $X$
une $F$-surface projective, lisse, g\'eom\'etriquement connexe, 
 satisfaisant  $H^1(X,O_{X})=0$. Supposons le groupe
  $\Pic({\overline X})=\NS({\overline X})$ sans torsion. On a alors un homomorphisme surjectif
$$ H^3_{nr}(X,\Q/\Z(2)) \to   \Coker[CH^2(X) \to CH^2({\overline X})^G].$$
Si l'indice $I(X)$ de $X$, qui est le pgcd des degr\'es sur $F$ des points ferm\'es de $X$,
n'est pas \'egal \`a 1, alors $H^3_{nr}(X,\Q/\Z(2)) \neq 0$.
\end{prop}

 \begin{proof}  Sous les  hypoth\`eses de la proposition,  le groupe $H^0({\overline X},{\mathcal K}_{2})$
 est uniquement divisible \cite[Cor. 1.12]{CTRaskind}. 
 Le groupe $H^1({\overline X},{\mathcal K}_{2})$ est extension d'un groupe fini par un groupe divisible
 \cite[Thm. 2.2]{CTRaskind} et donc $H^2(G, H^1({\overline X},{\mathcal K}_{2}))=0$.  Comme $X$ est une surface,
 on a $ H^3_{nr}({\overline X},\Q/\Z(2)) =0$.
  La surjection r\'esulte alors
 de la proposition \ref{longuesuitegenerale1} (ou de la proposition \ref{longuesuitegenerale2}). 
 Pour la surface $X$, on a la suite exacte de modules galoisiens
$$ 0 \to A_{0}({\overline X}) \to CH^2({\overline X}) \to \Z \to 0,$$
o\`u la fl\`eche $CH^2({\overline X}) \to \Z$ est donn\'ee par le degr\'e des z\'ero-cycles.
L'hypoth\`ese $ H^1(X,O_{X})=0$ implique que le groupe $ A_{0}({\overline X})$ est uniquement divisible
(th\'eor\`eme de Roitman). L'application induite
$CH^2({\overline X})^G \to \Z$ est donc surjective, et le
 groupe $\Coker[CH^2(X) \to CH^2({\overline X})^G]$
 a pour quotient le groupe $\Z/I(X)$.
 \end{proof}
 
 \begin{exemple}
Soit  $F=\C((t))$. Soient   $n>0$ un entier et  $X \subset \P^3_{F}$ la surface d\'efinie par l'\'equation
homog\`ene
 $$ x_{0}^n +tx_{1}^n +t^2x_{2}^n+ t^3x_{3}^n=0.$$
D'apr\`es  \cite[Prop. 4.4]{ELW}, pour  $n=4$ (surface $K3$) et pour $n$ premier \`a $6$,
on a  $I(X)\neq 1$. La proposition ci-dessus donne alors  $H^3_{nr}(X,\Q/\Z(2)) \neq 0$.
  \end{exemple}

\section{Vari\'et\'es \`a petit motif sur un corps non alg\'ebriquement clos}\label{principal}

Commen\c cons par un \'enonc\'e g\'en\'eral mais peut-\^etre un peu lourd.
 
\begin{theo}\label{BMgeneralise}
Soit $X$ une $F$-vari\'et\'e  projective, lisse et  g\'eo\-m\'etri\-que\-ment int\`egre.
 
Supposons satisfaites les   conditions :

(i) Sur un domaine universel $\Omega$, le degr\'e  $ CH_{0}(X_{\Omega}) \to \Z$
est un isomorphisme.

(ii) Le groupe $\Pic({\overline X})= \NS({\overline X})$  est sans torsion.

(iii)  Pour tout $\ell$ premier, le groupe $H^{3}_{\et}({\overline X},\Z_{\ell}) $ est
 sans torsion. 

(iv) On a au moins l'une des propri\'et\'es :  $X(F) \neq \emptyset$ ou  $\cd(F) \leq 3$.

Alors on a une suite exacte
  $$ \Ker[CH^2(X) \to CH^2({\overline X})^G] {\buildrel \alpha \over \longrightarrow} H^1(G,\Pic({\overline X}) \otimes {\FF}^{\times}) \to  \hskip3cm $$$$
   \Ker[ H^3_{nr}(X,\Q/\Z(2))/H^3(F,\Q/\Z(2)) \to  H^3_{nr}({\overline X},\Q/\Z(2))]   \to
$$$$  \hskip3cm  \Coker[CH^2(X) \to CH^2({\overline X})^G] 
{\buildrel \beta \over \longrightarrow}
  H^2(G,\Pic({\overline X}) \otimes {\FF}^{\times}).$$

Sous l'hypoth\`ese $X(F) \neq \emptyset$ ou $\cd(F)\leq 2$,
la fl\`eche $\alpha$
est injective.
 \end{theo}
 
 \begin{proof} Comme on a suppos\'e ${\rm car}(F)=0$,
l'hypoth\`ese (i)
 implique \cite{BS}
 que tous les groupes $H^{i}(X,\mathcal{O}_{X})$ pour $i \geq 1$
 sont nuls,  que   l'on a $\Pic({\overline X})= \NS({\overline X})$, et que le groupe de Brauer
 $\Br({\overline X})$ s'identifie au groupe fini  $ \oplus_{\ell}H^{3}({\overline X},\Z_{\ell})_{tors}$.
 Sous l'hypoth\`ese (i), l'hypoth\`ese (iii) est donc \'equivalente \`a  $\Br({\overline X})=0$.

Sous les hypoth\`eses (i) et (iii), la proposition \ref{H1Giso} donne 
$$H^{i}(G,\Pic({\overline X})\otimes {\FF}^{\times}) \oi H^{i}(G,H^1({\overline X},{\mathcal K}_{2}))$$
pour tout $i\geq 1$.

Sous les hypoth\`eses (i) et (ii), d'apr\`es \cite[Prop. 1.14]{CTRaskind},
on a $K_{2}{\FF}=H^0({\overline X},{\mathcal K}_{2})$. Le groupe $K_{2}{\FF}$ \'etant uniquement divisible,
on peut appliquer la Proposition \ref{longuesuitegenerale1}  (ou la proposition \ref{longuesuitegenerale2}).
\end{proof}

\begin{cor}\label{corollaireprincipal}\label{corprinc}
Soit $X$ une $F$-vari\'et\'e  projective, lisse et g\'eom\'etriquement int\`egre.

Supposons $X(F) \neq \emptyset$  ou $\cd(F) \leq 3$.

Supposons satisfaite  l'une des hypoth\`eses suivantes :

(i) la vari\'et\'e $\overline{X}$ est rationnelle;

(ii) la vari\'et\'e $\overline{X}$ est  rationnellement connexe, $\Br(\overline{X})=0$ et $H^{3}_{nr}({\overline X},\Q/\Z(2))=0$;

(iii) la vari\'et\'e $\overline{X}$ est de dimension 3, rationnellement connexe, et $\Br(\overline{X})=0$;

(iv) la vari\'et\'e $\overline{X}$ est de dimension 3, unirationnelle, et $\Br(\overline{X})=0$.

Alors on a une suite exacte
  $$ \Ker[CH^2(X) \to CH^2({\overline X})^G] {\buildrel \alpha \over \longrightarrow}  H^1(G,\Pic({\overline X}) \otimes {\FF}^{\times}) \to  \hskip3cm$$$$ 
   H^3_{nr}(X,\Q/\Z(2))/H^3(F,\Q/\Z(2)) \to $$
$$  \Coker[CH^2(X) \to CH^2({\overline X})^G]  {\buildrel \beta \over \longrightarrow}    H^2(G,\Pic({\overline X}) \otimes {\FF}^{\times}).$$
 
 Sous l'hypoth\`ese $X(F) \neq \emptyset$ ou $\cd(F)\leq 2$,
la fl\`eche  $\alpha$
est injective.
\end{cor}

\begin{proof} Le cas (iv) est un cas particulier du cas (iii). Sous l'hypoth\`ese (i), tous les groupes
$H^{i}_{nr}({\overline X},\Q/\Z(2))$ sont nuls pour $i\geq 1$. 
Pour $i=1$, cela \'etablit que $\Pic({\overline X})$ est sans torsion et donc  $\Pic({\overline X})= \NS({\overline X})$.
Pour $i=2$, cela \'etablit $\Br({\overline X})=0$ et donc  $H^{3}_{\et}({\overline X},\Z_{\ell}) _{tors}=0$ pour tout premier $\ell$.

Sous l'hypoth\`ese (iii), on
a $H^{3}_{nr}({\overline X},\Q/\Z(2))=0$. Cette annulation vaut  en effet  pour
tout solide unir\'egl\'e  \cite[Cor. 6.2]{CTVoisin},
c'est un corollaire 
d'un th\'eor\`eme de C. Voisin.

 L'\'enonc\'e est alors
une cons\'equence  imm\'ediate du th\'eor\`eme \ref{BMgeneralise}.
\end{proof}

\begin{remas}\label{lienBM}

(a)
Dans le cas particulier o\`u  $X$ est une $F$-compactification lisse \'equivariante d'un $F$-tore,
le corollaire \ref{corprinc}  est tr\`es proche d'un r\'esultat de Blinstein et Merkurjev (\cite[Prop. 5.9]{BM}).
Dans ce cas, le groupe $CH^2({\overline X})$ est sans torsion, le groupe
$$\Ker[CH^2(X) \to CH^2({\overline X})^G]$$ co\"{\i}ncide donc avec $CH^2(X)_{tors}$.
Par ailleurs, l'intersection des cycles
$$ \Pic({\overline X}) \times  \Pic({\overline X})  \to CH^2({\overline X})$$
induit une application naturelle surjective (\cite[\S 5.2, Proposition, p. 106]{Fulton})
$${\rm Sym}^2(\Pic({\overline X})) \to CH^2({\overline X}).$$

\medskip

(b) Soit  $X$  une $F$-compactification lisse   d'un $F$-tore.
La  fl\`eche $$H^1(G,\Pic({\overline X}) \otimes {\FF}^{\times}) \to   
   H^3_{nr}(X,\Q/\Z(2))/H^3(F,\Q/\Z(2)) $$
  intervient dans l'\'etude de l'approximation faible pour $X$ sur le corps $F$ des fonctions
  d'une courbe sur un corps $p$-adique (Harari, Scheiderer, Szamuely \cite[Thm. 4.2]{HSS}).
  Pour $X$ une $F$-compactification lisse d'un espace principal homog\`ene d'un
  $F$-tore, il conviendrait de comparer la fl\`eche
   $$H^1(G,\Pic({\overline X}) \otimes {\FF}^{\times}) \to   
   H^3_{nr}(X,\Q/\Z(2))/H^3(F,\Q/\Z(2)) $$
ici obtenue (le corps $F$ satisfaisant $cd(F) \leq 3$)  avec l'application (19)
utilis\'ee dans  \cite[Thm. 5.1]{HS}.

\medskip

(c) Soit $X/F$ une surface projective, lisse, g\'eom\'etriquement rationnelle poss\'edant un z\'ero-cycle de
degr\'e 1, et telle que le module galoisien $\Pic({\overline X})$ soit un facteur direct d'un module
de permutation. 
 Le corollaire ci-dessus implique alors  $H^3(F,\Q/\Z(2)) \oi H^3_{nr}(X,\Q/\Z(2) )$.
 C'est un cas particulier d'une remarque g\'en\'erale pour toute telle surface. Si le module galoisien $\Pic({\overline X})$
 est un facteur direct d'un module  
 de permutation,
 alors, d'apr\`es \cite[Prop. 4, p.~12]{CTHilb}, sur tout corps $L$ contenant $F$,
 l'application degr\'e $CH_{0}(X_{L}) \to \Z$ est un isomorphisme.  Ceci implique 
 $H^i(F,\Q/\Z(2)) \oi H^i_{nr}(X,\Q/\Z(2) )$ pour tout entier $i$ (cas particulier d'un th\'eor\`eme de Merkurjev,
 cf. \cite[Thm. 1.4]{ACTP}).
\medskip

(d) Dans l'article \cite{CTMadore} avec Madore, on a construit des exemples de corps~$F$ de dimension cohomologique 1
et de surfaces $X/F$ projectives, lisses, g\'eom\'etriquement rationnelles sans z\'ero-cycle de degr\'e 1.
Pour de telles surfaces, le corollaire \ref{corprinc}
ci-dessus donne   $$ H^3_{nr}(X,\Q/\Z(2)) = H^3_{nr}(X,\Q/\Z(2))/H^3(F,\Q/\Z(2))  \neq 0.$$
\end{remas}

(e) Pour $X$ une $F$-vari\'et\'e projective, lisse, g\'eom\'etriquement connexe quelconque, chacun des trois groupes suivants
est un invariant $F$-birationnel de $X$ : 

le groupe
$ \Ker[CH^2(X) \to CH^2({\overline X})^G] $,

le groupe
  $H^1(G,\Pic({\overline X}) \otimes {\FF}^{\times})$, 
  
  le groupe
$ H^3_{nr}(X,\Q/\Z(2)).$

Si la dimension cohomologique de $F$ est au plus 1, le conoyau $   \Coker[CH^2(X) \to CH^2({\overline X})^G]$ 
est un invariant $F$-birationnel, comme on voit en consid\'erant la situation de l'\'eclatement en une sous-vari\'et\'e
lisse. En g\'en\'eral,
le conoyau $     \Coker[CH^2(X) \to CH^2({\overline X})^G]$ n'est pas un invariant birationnel, comme on
peut voir en \'eclatant 
$\P^3_{F}$ en
 une $F$-conique lisse sans $F$-point.
 Ceci montre aussi que l'application
$$\beta :  \Coker[CH^2(X) \to CH^2({\overline X})^G]  \longrightarrow H^2(G,\Pic({\overline X}) \otimes {\FF}^{\times})$$
n'est pas toujours nulle.

\section{Vari\'et\'es \`a petit motif sur le corps des complexes}\label{hypcubiques}\label{petitmotif}

\subsection{Rappels}

Pour tout corps $F$ contenant $\C$, on note  $A^2(X_{F})$ le 
sous-groupe de $CH^2(X_{F})$
form\'e des classes de cycles qui sur une cl\^oture alg\'ebrique $\FF$ de $F$ sont alg\'ebriquement
\'equivalents \`a z\'ero.

La proposition suivante rassemble des r\'esultats connus, utiles pour la suite
de ce paragraphe.
  
\begin{prop}\label{rappelblochsri}
 Soit $X$
une vari\'et\'e connexe, projective et lisse sur le corps des complexes.
Supposons que l'application degr\'e $CH_{0}(X) \to \Z$ est un isomorphisme.

Alors

(i) On a $H^{i}(X,\mathcal{O}_{X})=0$ pour $i \geq 1$.

(ii) Pour tout corps $F$ contenant $\C$, les applications de restriction
  $$\Pic(X) \to \Pic(X_F) \to \Pic(X_{\overline F})$$
  sont des isomorphismes, et $\Pic(X)=\NS^1(X) =H^2_{Betti}(X,\Z)$.

(iii) \'Equivalence homologique et \'equivalence alg\'ebrique co\"{\i}ncident sur le 
groupe de Chow $CH^2(X)$. 

(iv) Le quotient $\NS^2(X):=CH^2(X)/A^2(X) \subset H^4_{Betti}(X,\Z)$ est un groupe ab\'elien de type fini.
Pour tout corps alg\'ebriquement clos $F$ contenant $\C$, on a $\NS^2(X) \oi \NS^2(X_{F})$.

(v) Il existe une vari\'et\'e ab\'elienne $B$ sur $\C$
 qui est un repr\'esentant alg\'ebrique
de $A^2(X)$, au sens de Murre  (\cite{murre}, cf.  \cite[D\'ef. 3.2]{Beauville}). 
Pour  tout corps   $F$ contenant $\C$, on a un 
homomorphisme $A^2(X_{F}) \to  B(F)$ fonctoriel en $F$, et cet homomorphisme
est un isomorphisme si $F$ est alg\'ebriquement clos.

(vi) S'il existe un premier $l$ avec $H^3_{Betti}(X,\Z/l)=0$,
alors $A^2(X)=0$,   on a une inclusion $CH^2(X) \hookrightarrow H^4_{Betti}(X,\Z)$,
et ces groupes sont sans $l$-torsion.

\end{prop}

\begin{proof}
Pour les \'enonc\'es (i), (iii), (iv), (v), dus  essentiellement \`a Bloch et Srinivas, et reposant
sur des th\'eor\`emes de Merkujev--Suslin et de Murre \cite{murre}, voir \cite[Thm. 1]{BS} et \cite{voisinlivre}.
L'\'enonc\'e (ii) est une cons\'equence connue de $H^1(X,\mathcal{O}_{X})=0$.
Le dernier \'enonc\'e de (iv) est une propri\'et\'e g\'en\'erale des quotients des groupes
de Chow modulo l'\'equivalence alg\'ebrique.
Pour l'\'enonc\'e (vi), les travaux de Bloch et de Merkurjev--Suslin
 montrent que
le sous-groupe   de $l$-torsion $CH^2(X)[l]$ de $CH^2(X)$ est un sous-quotient de $H^3_{Betti}(X,\Z/l)$.
On a donc  $CH^2(X)[l]=0$ et a fortiori $A^2(X)[l]=0$, donc $B[l]=0$,
donc la vari\'et\'e ab\'elienne $B$ est triviale et $A^2(X)=0$.  
\end{proof}

\begin{remas}
(a) Si $X$ est une vari\'et\'e rationnellement connexe, alors
  l'application degr\'e $CH_{0}(X) \to \Z$ est un isomorphisme, les propri\'et\'es (i) \`a (v) sont donc
  satisfaites.
  
(b)
Les \'enonc\'es (iii) \`a (v) valent sous l'hypoth\`ese plus faible  qu'il existe une courbe projective et lisse $C$ et un morphisme
$C \to X$ qui induise une surjection $CH_{0}(C) \to CH_{0}(X)$.
\end{remas}

\subsection{Cycle de codimension 2 universel}

Soit $F$ un corps, $X$ et $Y$ deux $F$-vari\'et\'es projectives, lisses, g\'eom\'etriquement connexes.
Soit $z \in CH^2(X\times_{F}Y)$ une classe de cycle de codimension 2.
La th\'eorie des correspondances  \cite{fultonchow} donne une application bilin\'eaire
$$ CH_{0}(Y) \times CH^2(Y\times_{F}X) \to CH^2(X).$$
Le sous-groupe $A_{0}(Y)$ des z\'ero-cycles de degr\'e 0
est form\'e de classes g\'eom\'etriquement alg\'ebriquement \'equivalentes \`a z\'ero.
Ainsi tout \'el\'ement  $z \in CH^2(Y\times_{F}X)$ d\'efinit  un homomorphisme
$$CH_{0}(Y) \to CH^2(X)$$
envoyant le groupe  $A_{0}(Y)$ dans le sous-groupe $A^2(X) \subset CH^2(X)$
d\'efini au d\'ebut du \S  \ref{petitmotif}. Cette application est fonctorielle en le corps de base $F$.
Via la fl\`eche \'evidente $Y(F) \to CH_{0}(Y)$ envoyant un point rationnel sur sa classe dans
le groupe de Chow, elle induit une application qui ne saurait \^etre qu'ensembliste
$$ Y(F) \to CH^2(X).$$ Si $Y$ est muni d'un point rationnel not\'e $O$,
en envoyant $P$ sur la classe de $P - O$, on d\'efinit une fl\`eche ensembliste
$$ \theta_{z} : Y(F) \to A^2(X)$$ envoyant $O$ sur $0$.

\medskip

Soient $X$ et $B$ comme dans la proposition \ref{rappelblochsri}.
On note $O$  l'\'el\'ement neutre de
 de $B(\C)$.
 La d\'efinition suivante est une variante de celle donn\'ee par Claire Voisin \cite[D\'ef. 0.5]{voisin}.
 
 \begin{defi}
Pour $X$ et $B$  comme ci-desssus, on dit 
qu'il existe un cycle de codimension 2 universel sur $X$
s'il existe un cycle $z \in CH^2(B \times X)$
tel que, sur tout corps $F$ contenant $\C$,
l' application
ensembliste 
$$\theta_{z} : B(F)  \to A^2(X_{F})$$
d\'efinie ci-dessus satisfasse la propri\'et\'e :
 
L'application compos\'ee
$$B(F) \to A^2(X_{F}) \to B(F)$$
est l'identit\'e sur $B(F)$.
\end{defi}

Le th\'eor\`eme ci-dessous est une variante d'un r\'esultat de C. Voisin \cite[Thm. 2.1, Cor. 2.3]{voisin}.
La d\'emonstration ici propos\'ee  diff\`ere sensiblement de celle donn\'ee   dans \cite{voisin}.

\begin{theo}\label{h3univtrivial}
Soit $X$
une vari\'et\'e  connexe, projective et lisse sur~$\C$.
Supposons  les conditions suivantes satisfaites.

(i) L'application degr\'e   $ CH_{0}(X) \to \Z$
est un isomorphisme. 

(ii) Les groupes $H^2_{Betti}(X,\Z)$ et $H^3_{Betti}(X,\Z)$ sont sans torsion.

(iii) On a $H^3_{nr}(X,\Q/\Z(2))=0$.

Alors :

(1) Pour tout corps $F$ contenant $\C$,  on a une suite exacte
 \[
 0 \to    H^3_{nr}(X_{F},\Q/\Z(2))/H^3(F,\Q/\Z(2))     \to  \hskip5cm \tag{5.2}\label{**} \]
 \[
 \Coker[CH^2(X_{F}) \to CH^2(X_{\overline F})^G] {\buildrel \beta \over \longrightarrow} H^2(G,\Pic(X_{\overline F}) \otimes {\FF}^{\times}).  \]

(2) Soit $B$ 
le repr\'esentant alg\'ebrique
de $A^2(X)$ (Prop. \ref{rappelblochsri} (v)).
S'il existe 
un cycle de codimension 2   universel  dans $CH^2(B \times X)$,
alors  pour tout corps $F$ contenant $\C$,
on a $H^3(F,\Q/\Z(2)) \oi H^3_{nr}(X_{F},\Q/\Z(2))$.
  \end{theo}
  
Note :   Sous l'hypoth\`ese $CH_{0}(X)=\Z$, la condition  $H^3_{nr}(X,\Q/\Z(2))=0$ est, d'apr\`es \cite[Thm. 1.1]{CTVoisin}, \'equivalente au fait
que la conjecture de Hodge enti\`ere vaut en degr\'e 4, i.e. pour les cycles de codimension 2.

 \begin{proof}
Soit $F$ un corps contenant $\C$.  
  Soit $\overline F$ une cl\^{o}ture alg\'ebrique de $F$ et $G=\Gal({\overline F}/F)$.
  D'apr\`es le th\'eor\`eme \ref{BMgeneralise} appliqu\'e \`a la $F$-vari\'et\'e $X_{F} : =X\times_{\C}F$,
on a une suite exacte
  $$H^1(G,\Pic(X_{\overline F}) \otimes {\FF}^{\times}) \to \hskip8cm $$$$ 
   \Ker[ H^3_{nr}(X_{F},\Q/\Z(2))/H^3(F,\Q/\Z(2)) \to  H^3_{nr}(X_{\overline F},\Q/\Z(2))]   \to
$$$$  \hskip3cm  \Coker[CH^2(X_{F}) \to CH^2(X_{\overline F})^G] {\buildrel \beta \over \longrightarrow} H^2(G,\Pic(X_{\overline F}) \otimes {\FF}^{\times}) $$

On sait \cite[Thm. 4.4.1]{santabarbara} que la cohomologie non ramifi\'ee est invariante
par extension de corps de base alg\'ebriquement clos. Sous l'hypoth\`ese (iii), on
a donc $H^3_{nr}(X_{\overline{F}},\Q/\Z(2))=0$.
Sous les hypoth\`eses (i) et (ii),
  les applications de restriction $\Pic(X) \to \Pic(X_{F}) \to \Pic(X_{\overline F})$ sont des isomorphismes
de r\'eseaux (Proposition  \ref{rappelblochsri} (ii)).
 L'action de
$\Gal({\overline F}/F)$ sur le r\'eseau $\Pic(X_{\overline F})$ est donc   triviale.
Le th\'eor\`eme 90 de Hilbert donne alors $$H^1(G,\Pic({\overline X}) \otimes {\FF}^{\times})=0.$$
Ceci donne la suite exacte (\ref{**}).

Supposons qu'il existe un cycle de codimension 2 universel. Alors, sur tout corps $F$ contenant $\C$,
on dispose de l'application ensembliste $B(F) \to A^2(X_{F})$ qui compos\'ee avec 
l'application $ A^2(X_{F}) \to B(F)$ est l'identit\'e.
Ceci implique que l'homomorphisme   $A^2(X_{F}) \to A^2(X_{\FF})^G$ est une surjection.
L'application compos\'ee  $\NS^2(X)  \to \NS^2(X_{\FF})^G$
est surjective, car $\NS^2(X) \to \NS^2(X_{\FF})$ est un isomorphisme (Prop. \ref{rappelblochsri} (iv)).
Ainsi
$CH^2(X) \to  CH^2(X_{\overline F})^G$ est surjectif, et de la suite exacte (\ref{**})  on
d\'eduit  $H^3(F,\Q/\Z(2))= H^3_{nr}(X_{F},\Q/\Z(2) ) $.
  \end{proof}

\begin{rema}
Sous des hypoth\`eses additionnelles, C. Voisin \cite[Thm. 2.1, Cor. 2.3]{voisin} \'etablit une r\'eciproque
  du th\'eor\`eme \ref{h3univtrivial}. Il serait souhaitable d'\'etablir une telle r\'eciproque par les m\'ethodes
  plus $K$-th\'eoriques du pr\'esent article, en utilisant la suite exacte (\ref{**})
  pour le corps des fonctions
  $F=\C(B)$ du repr\'esentant alg\'ebrique $B$ de $A^2(X)$.
\end{rema}

 \subsection{Troisi\`eme groupe de cohomologie non ramifi\'e des hypersurfaces de Fano}\label{hyperFano}

 \begin{theo}\label{hypersurface}
  Soit $n \geq 4$.
Soit $X \subset \P^n_{\C}$ une hypersurface lisse de degr\'e $d \leq n$.  

(i) La fl\`eche degr\'e $CH_{0}(X) \toÊ\Z$ est un isomorphisme.

(ii) On a $ \Pic(X) =\NS(X)= H^2_{Betti}(X,\Z)=\Z$, et ce groupe est
engendr\'e par la classe d'une section hyperplane.

(iii) Le groupe  $H^3_{Betti}(X, \Z)$  est sans torsion, et nul pour $n \geq 5$.

(iv) Pour  $n\geq 5$,   \'equivalences rationnelle, alg\'ebrique et homologique co\"{\i}ncident
sur les cycles de codimension 2 sur $X$, et
on a une injection de r\'eseaux  $CH^2(X) \hookrightarrow H^4_{Betti}(X, \Z)$.

(v) Pour   $n\neq  5$, $H^4_{Betti}(X, \Z)=\Z$, et l'application
  $$CH^2(X) \to H^4_{Betti}(X, \Z)=\Z$$est surjective, et
  est un isomorphisme pour $n > 5$.

(vi) Pour $n =4$ et $n>  5$, on a $H^3_{nr}(X,\Q/\Z(2))=0$.

   (vii)  Pour   $n\geq 5$, pour  tout corps $F$
 contenant $\C$, de cl\^oture alg\'ebrique $\FF$,
 avec $G:=\Gal(\FF/F)$,
  la fl\`eche naturelle
 $$CH^2(X_{F}) \to CH^2(X_{\FF})^G$$
   est surjective, et on a une suite exacte  naturellement scind\'ee
 $$ 0 \to H^3(F,\Q/\Z(2)) \to  
 H^3_{nr}(X_{F},\Q/\Z(2)) \to  H^3_{nr}(X_{\overline F},\Q/\Z(2)) \to 0.$$
 Pour $n > 5$,on a
 $$   H^3(F,\Q/\Z(2)) \oi H^3_{nr}(X_{F},\Q/\Z(2)).$$
 
 (viii) Pour $n=4$, soit
 $B$ le repr\'esentant alg\'ebrique de $A^2(X)$.
S'il existe un cycle universel de codimension 2 dans $CH^2(B \times X)$,
alors pour  tout corps $F$
 contenant $\C$, on a
$H^3(F,\Q/\Z(2)) \oi
 H^3_{nr}(X_{F},\Q/\Z(2))$,
   et l'application
$CH^2(X_{F}) \to CH^2(X_{\overline F})^G$ est   surjective.
\end{theo}

\begin{proof}

Les \'enonc\'es (i) \`a (v) sont bien connus. Comme ils sont utilis\'es pour \'etablir les points suivants,
donnons quelques rappels \`a leur sujet.

L'hypoth\`ese $d \leq n$ assure $CH_{0}(X)\oi\Z$, soit (i).  
C'est un th\'eor\`eme de Roitman, que l'on peut aussi voir comme un cas particulier
du th\'eor\`eme de Campana et  Koll\'ar-Miyaoka-Mori assurant qu'une vari\'et\'e de Fano est
 rationnellement connexe.  
L'\'enonc\'e (ii) vaut pour toute hypersurface lisse
dans $\P^n_{\C}$, $n \geq 4$.

Pour  $n\geq 5$,     les th\'eor\`emes de
Lefschetz donnent $H^3_{Betti}(X,\Z)=0$ et
 $H^3_{Betti}(X,\Z/l)=0$    pour tout $l$ premier.
L'\'enonc\'e (i) et  la proposition \ref{rappelblochsri} donnent  alors (iv).

 Pour $n=4$, $H^3_{Betti}(X,\Z)$ est sans torsion.
  Par ailleurs
 $H^4_{Betti}(X,\Z)=\Z$ (par dualit\'e de Poincar\'e),
 la restriction $\Z=H^4_{Betti}(\P^4,\Z) \to H^4_{Betti}(X,\Z)=\Z$
 est l'identit\'e sur $\Z$.
 
  Pour $n \geq 3$,  toute hypersurface $X \subset \P^n_{\C}$ de degr\'e $d \leq n$ contient une droite de $ \P^n_{\C}$.
C'est un r\'esultat classique mais d\'elicat  dans le cas $d=n$
(voir  \cite{Debarre}). Pour $d <n$,  cela r\'esulte d'un calcul  imm\'ediat de dimension, qui montre 
que  par tout point de $X$ il passe une droite de $ \P^n_{\C}$ contenue dans $X$.
 
 Soit $n=4$.  L'hypersurface $X$ contient une droite de $\P^4_{\C}$.
 La classe de cette droite dans  $CH^2(X)$ engendre donc
 $H^4_{Betti}(X,\Z)=\Z$.

Pour $n \geq 6$, les th\'eor\`emes de Lefschetz donnent que
  la fl\`eche de restriction $\Z= H^4_{Betti}(\P^n_{\C},\Z) \to H^4_{Betti}(X,\Z)$
est un isomorphisme.
Le diagramme commutatif
$$\begin{array}
{ccccccccc}
 {CH}^{2}(X) &   \hookrightarrow & H^4_{Betti}(X,\Z) \cr
\uparrow & & \uparrow \cr
{CH}^2(\P^n_{\C}) & \oi & H^4_{Betti}(\P^n_{\C},\Z)
\end{array}
$$
donne alors   $CH^2(X) \oi H^4_{Betti}(X,\Z) =\Z$,
la conjecture de Hodge enti\`ere en degr\'e 4 vaut donc pour $X$, et
  la th\'eorie de Bloch-Ogus ou     \cite[Thm. 1.1]{CTVoisin}
donnent alors $H^3_{nr}(X, \Q/\Z(2))=0$ soit (vi) pour $n \geq 6$.
Le m\^eme argument vaut pour $n=4$ et $d\leq 4$, puisque l'application 
$CH^2(X) \to  H^4_{Betti}(X,\Z) =\Z$ est surjective. Ceci \'etablit (v) et (vi).
Pour $n=4$,  (vi)  est un cas particulier 
d'un r\'esultat de C. Voisin \cite[Cor. 6.2]{CTVoisin}.

  \'Etablissons les points (vii) et (viii).

Pour tout $n \geq 4$,  Pour tout corps $F$ contenant~$\C$, on a
$$\Pic(X)=\Pic(X_{F})=\Z,$$
le groupe \'etant engendr\'e par la classe d'une section hyperplane (th\'eor\`eme de Max Noether).
On a donc  $H^1(G, \Pic(X_{\FF}) \otimes \FF^{\times})= H^1(G, \FF^{\times})=0$ (Thm. 90 de Hilbert).
Les \'enonc\'es d\'ej\`a \'etablis et le  le th\'eor\`eme \ref{BMgeneralise} donnent  alors une suite exacte
  \begin{equation}
  0 \to  H^3(F,\Q/\Z(2)) \to \Ker[H^3_{nr}(X_{F},\Q/\Z(2)) \to H^3_{nr}(X_{\FF},\Q/\Z(2))]\to  \label{***}
  \end{equation}
$$
  \hskip2cm  \Coker[CH^2(X_{F}) \to CH^2(X_{\overline F})^G]   {\buildrel  \beta \over \rightarrow}  H^2(G, \Pic(X_{\FF}) \otimes \FF^{\times}).
$$
 Pour $n \geq 5$,    d'apr\`es (iv),
  \'equivalence rationnelle et \'equivalence alg\'ebrique sur les cycles de codimension 2 de $X$
   co\"{\i}ncident sur un corps alg\'ebriquement clos.
Pour une vari\'et\'e projective,  lisse, connexe, sur un corps alg\'ebriquement clos,
les groupes d'\'equivalence de cycles modulo l'\'equivalence alg\'ebrique sont,
comme c'est bien connu et facile \`a \'etablir, invariants par extension du corps
de base \`a un autre corps alg\'ebriquement clos.
Ainsi  la fl\`eche compos\'ee $$ CH^2(X) \to CH^2(X_{F}) \to CH^2(X_{\overline F}) $$
  est l'identit\'e, donc  l'application $ CH^2(X_{F}) \to CH^2(X_{\overline F})^G$ 
  est surjective. On obtient donc dans ce cas une suite exacte
  $$0 \to  H^3(F,\Q/\Z(2)) \to H^3_{nr}(X_{F},\Q/\Z(2)) \to H^3_{nr}(X_{\FF},\Q/\Z(2)).$$
  D'apr\`es  \cite[Thm. 4.4.1]{santabarbara},  on a 
  $ H^3_{nr}(X,\Q/\Z(2))  \oi H^3_{nr}(X_{\FF},\Q/\Z(2))$, si  bien que la suite ci-dessus
  se  compl\`ete en une suite exacte naturellement scind\'ee
  $$0 \to  H^3(F,\Q/\Z(2)) \to H^3_{nr}(X_{F},\Q/\Z(2)) \to H^3_{nr}(X_{\FF},\Q/\Z(2)) \to 0.$$
   Pour $n > 5$, une application de (vi)  
   ach\`eve alors  d'\'etablir l'\'enonc\'e (vii).

   Pour $n=4$, on a d\'ej\`a \'etabli $H^3_{nr}(X,\Q/\Z(2))=0$ et
  $$H^3(F,\Q/\Z(2)) \oi H^3_{nr}(X_{F},\Q/\Z(2))$$ pour tout $F$ contenant $\C$.
  La deuxi\`eme partie de    l'\'enonc\'e (viii) r\'esulte alors
 de la suite exacte (\ref{***})  et
   du lemme \ref{restricdroite} (b) ci-apr\`es.
\end{proof}

 \begin{lem}\label{restricdroite}
 Soient  $n \geq 4$ et  $X \subset \P^n_{\C}$ une hypersurface
 lisse de degr\'e $d\leq n$.
 
 (a) Pour tout corps $F$ contenant $\C$, la fl\`eche naturelle
 $$\Pic(X_{\overline{F}}) \otimes \overline{F}^{\times}  \to    H^1(X_{\overline{F}} ,\K_{2}   )$$
 est un isomorphisme
 $$\overline{F}^{\times} \oi H^1(X_{\overline{F}} ,\K_{2}   ).$$
 
 (b) On a un isomorphisme  
  $$ \Ker [H^3_{nr}(X_{F},\Q/\Z(2))/H^3(F,\Q/\Z(2))  \to H^3_{nr}(X_{\FF},\Q/\Z(2)) ] \oi 
   \Coker[CH^2(X_{F}) \to CH^2(X_{\overline F})^G] .$$
  \end{lem}

\begin{proof}
 Pour tout corps $F$ contenant~$\C$, on a
$$\Pic(X)=\Pic(X_{F})=\Z,$$
le groupe \'etant engendr\'e par la classe d'une section hyperplane (th\'eor\`eme de Max Noether).
Comme on a  $CH_{0}(X)\oi\Z$ et que les groupes $H^3_{Betti}(X,\Z)$ sont sans torsion,
d'apr\`es  \cite[Thm. 2.12; Prop. 2.15]{CTRaskind},
la fl\`eche naturelle  
 $$\Pic(X_{\overline{F}}) \otimes \overline{F}^{\times}  \to    H^1(X_{\overline{F}} ,\K_{2}   )$$
est surjective.

On a vu ci-dessus que $X$ contient une droite de $\P^n$, soit 
$Y \subset X \subset \P^n_{\C}$.
 La restriction $$\Z=\Pic(X_{\overline{F}})  \to \Pic(Y_{\overline{F}})=\Z $$ est l'identit\'e sur $\Z$,
 car le groupe $\Pic(X_{\overline{F}}) $ est engendr\'e par la classe d'une section hyperplane.
 Donc la fl\`eche
$$\Pic(X_{\overline{F}}) \otimes \overline{F}^{\times} \to \Pic(Y_{\overline{F}}) \otimes \overline{F}^{\times}$$
est un isomorphisme.
 Pour la droite $Y$, l'application
$$\FF^{\times} = \Pic(Y_{\overline{F}}) \otimes \overline{F}^{\times} \to H^1(Y_{\overline{F}},\K_{2})$$
est un isomorphisme.
L'inclusion $Y \subset X$ induit un diagramme commutatif 
\[\xymatrix{
\Pic(X_{\overline{F}})\otimes  \overline{F}^{\times}
\ar[d]  \ar[r]   &   \Pic(Y_{\overline{F}})\otimes  \overline{F}^{\times} \ar[d]  \\
H^1(X_{\overline{F}},\K_{2})   \ar[r]  & H^1(Y_{\overline{F}},\K_{2})
}\]
 dans lequel la fl\`eche horizontale sup\'erieure est un isomorphisme,
la  fl\`eche verticale de droite aussi, et la  fl\`eche verticale de gauche est surjective.
La fl\`eche $\Pic(X_{\overline{F}})\otimes  \overline{F}^{\times} \to H^1(X_{\overline{F}},\K_{2})$
est donc un isomorphisme   $\overline{F}^{\times} \oi H^1(X_{\overline{F}},\K_{2})$,
ce qui \'etablit (a) et montre que la fl\`eche de restriction
$$H^1(X_{\overline{F}},\K_{2})  \to  H^1(Y_{\overline{F}},\K_{2})$$
est un isomorphisme.
 
Consid\'erons la suite exacte (\ref{***}).
 Pour $n \geq 5$, nous avons \'etabli $\Coker[CH^2(X_{F}) \to CH^2(X_{\overline F})^G] =0$,
 et donc  la fl\`eche
$$\beta :   \Coker[CH^2(X_{F}) \to CH^2(X_{\overline F})^G]   {\buildrel  \beta \over \rightarrow}  H^2(G, \Pic(X_{\FF}) \otimes \FF^{\times})$$
dans cette suite est nulle.
  
 Montrons que l'on a encore $\beta=0$ dans le cas $n=4$.
Nous avons ici recours au point de vue motivique,
i.e. \`a la proposition \ref{longuesuitegenerale2}.
L'application $\beta$ 
   est  
 induite par l'application compos\'ee
  $$CH^2(X_{\overline F})^G \to \bH^4(X_{\overline{F}},\Z(2))^G \to H^2(G, \bH^3(X_{\overline{F}},\Z(2))).$$
Chacune des deux applications  intervenant ici est d\'efinie pour toute vari\'et\'e lisse $X$,
et  leur formation est fonctorielle en la vari\'et\'e lisse $X$ : la seconde application 
vient de la suite spectrale consid\'er\'ee \`a la section \ref{motivique}.

 Soit $Y \subset X \subset \P^n_{\C} $  une droite. Comme
 la restriction
$$H^1(X_{\overline{F}},\K_{2}) \to H^1(Y_{\overline{F}},\K_{2})$$
est un isomorphisme, la fl\`eche $$ \beta : \Coker[CH^2(X_{F}) \to CH^2(X_{\overline F})^G]  \to H^2(G, H^1(X_{\overline F},\K_{2}))
  $$ se factorise par 
  $$\Coker[CH^2(Y_{F}) \to CH^2(Y_{\overline F})^G]=0$$ et  donc
est nulle,  ce qui via la suite exacte (\ref{***}) \'etablit l'\'enonc\'e (b).
 \end{proof}

Pour les  hypersurfaces cubiques, un r\'esultat de Claire Voisin permet
de compl\'eter le th\'eor\`eme \ref{hypersurface} dans le cas $n=5$.

\begin{theo}\label{h3nrcubiqueC}
Soit $X \subset \P^n_{\C}$, $n \geq 4$ une hypersurface cubique lisse.

(i) On  a  $H^3_{nr}(X,\Q/\Z(2))=0$.

(ii) Pour tout entier $n\geq 5$, 
   pour tout corps 
$F$  contenant $\C$, la fl\`eche
 $$ H^3(F,\Q/\Z(2)) \to H^3_{nr}(X_{F},\Q/\Z(2))$$
 est un isomorphisme, 
 et  l'application
 $$CH^2(X_{F}) \to CH^2(X_{\FF})^G$$
 est surjective.

(iii) 
 Pour $n=4$,  soit $B$ 
le repr\'esentant alg\'ebrique
de $A^2(X)$ (Prop. \ref{rappelblochsri} (v)).
S'il  existe
un cycle universel de codimension 2 dans $CH^2(B \times X)$,
alors  pour tout corps $F$ contenant $\C$
on a $H^3(F,\Q/\Z(2)) \oi H^3_{nr}(X_{F},\Q/\Z(2))$, et l'application
$CH^2(X_{F}) \to CH^2(X_{\overline F})^G$ est surjective.
\end{theo}

\begin{proof}
Pour $n \neq 5$, ceci est un cas particulier du th\'eor\`eme~\ref{hypersurface}.
Soit donc $n=5$. 
C.~Voisin 
a \'etabli  la conjecture de Hodge enti\`ere en degr\'e 4 pour  toute hypersurface cubique lisse $X \subset \P^5_{\C}$
\cite{voisinuniregle}, \cite[Thm. 0.4, Thm. 2.11]{voisindecdiag}.
D'apr\`es le th\'eor\`eme    \cite[Thm. 1.1]{CTVoisin},    ceci implique $H^3_{nr}(X,\Q/\Z(2))=0$,
et donc,  d'apr\`es \cite[Thm. 4.4.1]{santabarbara}, $H^3_{nr}(X_{\FF},\Q/\Z(2))=0$ pour tout corps alg\'ebriquement clos $\FF$
 contenant $\C$. L'\'enonc\'e (ii) est alors une cons\'equence du th\'eor\`eme 
 \ref{hypersurface} (vii).
  \end{proof}
  
  \begin{rema}
 Soit $n=5$.
 Si  l'hypersurface cubique $X \subset \P^5_{\C}$ contient un plan, on peut 
fibrer $X$ en quadriques au-dessus du plan. L'\'enonc\'e (i) r\'esulte alors de
 \cite[Cor. 8.2]{CTVoisin}, qui repose seulement sur
le calcul de la cohomologie non ramifi\'ee des quadriques de dimension 2
sur un corps quelconque (cas particulier des r\'esultats de Kahn, Rost, Sujatha sur les quadriques
de dimension quelconque).

Pour les {\it  hypersurfaces cubiques lisses  $X \subset \P^5_{\C}$  tr\`es g\'en\'erales
 contenant un plan}, l'isomorphisme
 $$ H^3(F,\Q/\Z(2)) \oi  H^3_{nr}(X_{F},\Q/\Z(2))$$
 dans la proposition \ref{h3nrcubiqueC}~(ii)  fut d'abord  \'etabli par des 
 m\'ethodes de $K$-th\'eorie et de formes quadratiques, en collaboration avec
  Auel et Parimala  \cite{ACTP}.
Il fut ensuite \'etabli  pour {\it  toute hypersurface cubique  lisse} $X \subset \P^5_{\C}$
par C.~Voisin  \cite[Thm. 2.1, Example 2.2]{voisin}, par une m\'ethode diff\'erente
de celle propos\'ee ici.

\end{rema}

 \begin{rema}
 Soit $n=4$.  Si pour une hypersurface cubique $X \subset \P^4_{\C}$ et
 un corps $F$ on  avait  $ H^3_{nr}(X_{F},\Q/\Z(2)) \neq H^3(F,\Q/\Z(2))$, alors
 $X$ ne serait pas stablement rationnelle. Un tel exemple n'est pas connu.
Dans     \cite[Thm. 4.5]{voisincubique},  
C. Voisin montre qu'il existe des hypersurfaces cubiques dans $\P^4_{\C}$
pour lesquelles
 le groupe de Chow des
z\'ero-cycles est universellement trivial, r\'esultat plus fort que
  $ H^3_{nr}(X_{F},\Q/\Z(2)) = H^3(F,\Q/\Z(2))$ pour tout $F$.

\end{rema}

\end{document}